\documentclass[12pt]{amsart}
\usepackage{amscd,amssymb,amsmath,float}
\newtheorem{thm}[equation]{Theorem}
\numberwithin{equation}{section}

\newtheorem{lem}[equation]{Lemma}

\newtheorem{prop}[equation]{Proposition}

\begin{document}
\raggedbottom \voffset=-.7truein \hoffset=0truein \vsize=8truein
\hsize=6truein \textheight=8truein \textwidth=6truein
\baselineskip=18truept
\renewcommand\arraystretch{1.5}
\def\mapright#1{\ \smash{\mathop{\longrightarrow}\limits^{#1}}\ }
\def\ss{\smallskip}
\def\ssum{\sum\limits}
\def\dsum{{\displaystyle{\sum}}}
\def\la{\langle}
\def\ra{\rangle}
\def\on{\operatorname}
\def\u{\on{U}}
\def\lg{\on{lg}}
\def\a{\alpha}
\def\bz{{\Bbb Z}}
\def\eps{\epsilon}
\def\br{{\bold R}}
\def\bc{{\bold C}}
\def\bN{{\bold N}}
\def\nut{\widetilde{\nu}}
\def\tfrac{\textstyle\frac}
\def\product{\prod}
\def\b{\beta}
\def\G{\Gamma}
\def\g{\gamma}
\def\zt{{\Bbb Z}_2}
\def\zth{{\bold Z}_2^\wedge}
\def\bs{{\bold s}}
\def\bg{{\bold g}}
\def\bof{{\bold f}}
\def\bq{{\bold Q}}
\def\be{{\bold e}}
\def\line{\rule{.6in}{.6pt}}
\def\xb{{\overline x}}
\def\xbar{{\overline x}}
\def\ybar{{\overline y}}
\def\zbar{{\overline z}}
\def\ebar{{\overline \be}}
\def\nbar{{\overline n}}
\def\fbar{{\overline f}}
\def\Ubar{{\overline U}}
\def\et{{\widetilde e}}
\def\ni{\noindent}
\def\ms{\medskip}
\def\ahat{{\hat a}}
\def\bhat{{\hat b}}
\def\chat{{\hat c}}
\def\nbar{{\overline{n}}}
\def\minp{\min\nolimits'}
\def\N{{\Bbb N}}
\def\Z{{\Bbb Z}}
\def\Q{{\Bbb Q}}
\def\R{{\Bbb R}}
\def\C{{\Bbb C}}
\def\el{\ell}
\def\o{\on{o}}
\def\lg{\on{lg}}
\def\dstyle{\displaystyle}
\def\ds{\dstyle}
\def\Remark{\noindent{\it  Remark}}
\title[$p$-adic Stirling numbers]
{$p$-adic Stirling numbers of the second kind}
\author{Donald M. Davis}
\address{Department of Mathematics, Lehigh University\\Bethlehem, PA 18015, USA}
\email{dmd1@lehigh.edu}
\date{July 29, 2013}


\maketitle
\begin{abstract} Let $S(n,k)$ denote the Stirling numbers of the second kind. We prove that the
$p$-adic limit of $S(p^ea+c,p^eb+d)$ as $e\to\infty$ exists for any integers $a$, $b$, $c$, and $d$
with $0\le b\le a$. We call the limiting $p$-adic integer $S(p^\infty a+c,p^\infty b+d)$.
When $a\equiv b$ mod $(p-1)$ or $d\le0$, we express them
in terms of $p$-adic binomial coefficients
$\binom{p^\infty\a-1}{p^\infty\b}$  introduced in a recent paper.
\end{abstract}
\section{Main theorems}\label{intro}
In \cite{Dinf}, the author defined, for integers $a$, $b$, $c$, and $d$, with $0\le b\le a$, $\binom{p^\infty a+c}{p^\infty b+d}$
to be the $p$-adic integer which is the $p$-adic limit of $\binom{p^ea+c}{p^eb+d}$, and gave explicit formulas for these in terms
of rational numbers and $p$-adic integers which, if $p$ or $n$ is even, could be considered to be $\u_p((p^\infty n)!):=\ds\lim_e\u_p((p^en)!)$. Here and throughout,
$\nu_p(-)$ denotes the exponent of $p$ in an integer or rational number and $\u_p(n)=n/p^{\nu_p(n)}$ denotes the unit factor in $n$.
Here we do the same for Stirling numbers $S(n,k)$ of the second kind; i.e., we prove that the $p$-adic limit of $S(p^e a+c,p^e b+d)$
exists, and call it $S(p^\infty a+c,p^\infty b+d)$. If $a\equiv b$ mod $(p-1)$ or $d\le0$, we express these  explicitly in terms
of certain $\binom{p^\infty\a-1}{p^\infty\b}$ together with certain Stirling-like rational numbers.

We now list our four main theorems, which will be proved in Sections \ref{sec2} and \ref{sec4}.
Let $\Z_p$ denote the $p$-adic integers with the usual metric.
\begin{thm}\label{thm1} Let $p$ be a prime, and $a$, $b$, $c$, and $d$  integers with $0\le a\le b$. Then the $p$-adic limit of
$S(p^ea+c,p^eb+d)$ exists in $\Z_p$. We denote the limit as $S(p^\infty a+c,p^\infty b+d)$.
\end{thm}

\begin{thm}\label{thm2} If $p$ is any prime and $0\le b\le a$, then $S(p^\infty a,p^\infty b)=0$ if $a\not\equiv b\mod(p-1)$, while
$$S(p^\infty a,p^\infty b)=\binom{p^\infty\frac{pa-b}{p-1}-1}{p^\infty\frac{p(a-b)}{p-1}}\text{ if }a\equiv b\mod(p-1).$$
\end{thm}
\noindent These $p$-adic binomial coefficients are as introduced in \cite{Dinf}.

Let $|s(n,k)|$ denote the unsigned Stirling numbers of the first kind.\begin{thm}\label{thm3}
If $0\le b\le a$, then
$$S(p^\infty a+c,p^\infty b+d)=\begin{cases}0&d=0,\ c\ne0\\
0&d<0,\ c\ge0\\
|s(|d|,|c|)|S(p^\infty a,p^\infty b)&c<0,\ d<0.\end{cases}$$
\end{thm}
\noindent In particular, if $a\not\equiv b$ mod $(p-1)$, then $S(p^\infty a+c,p^\infty b+d)=0$ whenever $d\le0$.

For any prime number $p$, integer $n$, and nonnegative integer $k$, define the partial Stirling numbers $T_p(n,k)$  (\cite{D}) by
\begin{equation}T_p(n,k)=\frac{(-1)^{k}}{k!}\sum_{i\not\equiv0\ (p)}(-1)^i\binom ki i^n.\label{Tdef}\end{equation}
\begin{thm}\label{thm4} If $a\equiv b\mod(p-1)$ and $d\ge 1$, then
$$S(p^\infty a+d-1,p^\infty b+d)= T_p(d-1,d) \binom{p^\infty\frac{pa-b}{p-1}-1}{p^\infty b}.$$
\end{thm}

When $a\equiv b$ mod $(p-1)$, results for all $S(p^\infty a+c,p^\infty b+d)$ with $d>0$ follow from these results and the standard formula
\begin{equation}\label{stirfor}S(n,k)=kS(n-1,k)+S(n-1,k-1).\end{equation}
Explicit formulas are somewhat complicated and are relegated to Section \ref{messysec}.
In Section \ref{sec5} we briefly mention
another version of $p$-adic Stirling numbers of the second kind.

\section{Proofs when $a\equiv b$ mod $(p-1)$ or $d\le0$}\label{sec2}
In this section, we prove Theorems  \ref{thm2}, \ref{thm3}, and \ref{thm4}.
If $a\equiv b$ mod $(p-1)$ or $d\le0$, Theorem \ref{thm1} follows immediately from Theorems \ref{thm2}, \ref{thm3},
 and \ref{thm4} and their proofs. These give explicit values for the limits when $d\le0$ and for at least
one value of $c$ when $d>0$. The existence of the limit for other values of $c$ follows from (\ref{stirfor})
and induction. We will prove Theorem \ref{thm1} when $a\not\equiv b$ mod $(p-1)$ and $d>0$ in Section \ref{sec4}.

We rely heavily on the following two results of Chan and Manna.
\begin{thm}$($\cite[4.2,5.2]{CM}$)$ Suppose $n> p^mb$ with $m\ge3$ if $p=2$. Then, mod $p^{m-1}$ if $p=2$, and mod $p^m$ if $p$ is odd,
$$S(n,p^mb)\equiv\begin{cases}\binom{n/2-2^{m-2}b-1}{n/2-2^{m-1}b}&\text{if $p=2$ and $n\equiv0 \mod 2$}\\
\binom{(n-p^{m-1}b)/(p-1)-1}{(n-p^mb)/(p-1)}&\text{if $p$ is odd and $n\equiv b\mod (p-1)$}\\
0&\text{otherwise.}\end{cases}$$\label{CMthm}
\end{thm}
\begin{thm}$($\cite[4.3,5.3]{CM}$)$ Let $p$ be any prime, and suppose $n\ge p^eb+d$. Then
$$S(n,p^eb+d)\equiv\sum_{j\ge0}S(p^eb+(p-1)j,p^eb)S(n-p^eb-(p-1)j,d)\mod{p^e}.$$\label{CM2}
\end{thm}

\begin{proof}[Proof of Theorem \ref{thm2}] The result follows from Theorem \ref{CMthm}.  If $p$ is odd and $a\not\equiv b$ mod $(p-1)$, then
$\nu_p(S(p^ea,p^eb))\ge e$, while if $a\equiv b$ mod $(p-1)$, then $$S(p^ea,p^eb)\equiv\binom{p^{e-1}\frac{pa-b}{p-1}-1}{p^{e-1}\frac{p(a-b)}{p-1}}\mod {p^e}.$$
If $p=2$, then
$$S(2^ea,2^eb)\equiv\binom{2^{e-2}(2a-b)-1}{2^{e-2}(2a-2b)}\mod {2^{e-1}}.$$
\end{proof}

Let $d_p(n)$ denote the sum of the digits in the $p$-ary expansion of a positive integer $n$.
\begin{proof} [Proof of Theorem \ref{thm3}]
The first case  follows readily  Theorem \ref{CMthm}. If $p=2$, this says that $\nu(S(2^ea+c,2^eb))\ge e-1$ if $c$ is odd, while if $c=2k$ is even, then, mod $2^{e-1}$,
$$S(2^ea+2k,2^eb)\equiv\binom{2^{e-1}a+k-2^{e-2}b-1}{2^{e-1}a+k-2^{e-1}b}.$$
If $0<k<2^{e-1}$, this has 2-exponent
$$\nu_2=d_2(a-b)+d_2(k)-(d_2(2a-b)+d_2(k-1))+d_2(2^{e-2}b-1)\to\infty$$
as $e\to\infty$, while if $k=-\ell<0$, then
$$\nu_2=e-1+d_2(a-b-1)-d_2(\ell-1)-(e-2+d_2(2a-b-1)-d_2(\ell))+d_2(2^{e-2}b-1)\to\infty.$$
The odd-primary case  follows similarly.

The second case of the theorem follows from the result for $c=0$  just established and (\ref{stirfor}) by induction.
For the third case, write $c=-k$ and $d=-\ell$ and argue by induction on $k$ and $\ell$, starting with the fact that the result is true if $k=0$ or $l=0$.
Then, mod $p^e$,
\begin{eqnarray*}S(p^ea-k-1,p^eb-\ell-1)&=&S(p^ea-k,p^eb-\ell)-(p^eb-\ell)S(p^ea-k-1,p^eb-\ell)\\
&\equiv&S(p^ea,p^eb)(|s(\ell,k)|+\ell |s(\ell,k+1)|)\\
&=&S(p^ea,p^eb)|s(\ell+1,k+1)|,\end{eqnarray*}
implying the result.\end{proof}

The proof of Theorem \ref{thm4} will utilize the following two lemmas. We let $\lg_p(x)=[\log_p(x)]$.
\begin{lem}\label{Tdiff} If $p$ is any prime and $k$ and $d$ are positive integers, then
$$\nu_p\bigl(T_p((p-1)k+d-1,d)-T_p(d-1,d)\bigr)\ge\nu_p(k)-\lg_p(d).$$
\end{lem}
\begin{proof} We have
\begin{eqnarray*}&&|T_p((p-1)k+d-1,d)-T_p(d-1,d)|\\
&=&\sum_{r=1}^{p-1}(-1)^r{\tfrac1{d!}}\sum_j(-1)^j\tbinom d{pj+r}(pj+r)^{d-1}((pj+r)^{(p-1)k}-1)\\
&=&\sum_{r=1}^{p-1}(-1)^r\sum_{i>0,t\ge0}r^{(p-1)k+d-1-i-t}\tbinom{(p-1)k}i\tbinom{d-1}t\tfrac1{d!}\dstyle\sum_j(-1)^j\tbinom d{pj+r}(pj)^{i+t}.
\end{eqnarray*}
Since $\binom{(p-1)k}i=\frac{(p-1)k}i\binom{(p-1)k-1}{i-1}$, we have $\nu_p\binom{(p-1)k}i\ge\nu_p(k)-\nu_p(i)$ for $i>0$.
Also $$\nu_p\bigl({\tfrac1{d!}}\sum_j(-1)^j\tbinom d{pj+r}(pj)^{i+t}\bigr)\ge\max(0,i+t-\nu_p(d!)),$$
with the first part following from \cite[Thm 1.1]{SD}. Thus it will suffice to show
$$\lg_p(d)-\nu_p(i)+\max(0,i+t-\nu_p(d!))\ge0.$$
This is clearly true if $\nu_p(i)\le\lg_p(d)$, while if $\nu_p(i)>\lg_p(d)=\ell$, then
$\nu_p(d!)\le\nu_p((p^{\ell+1}-1)!)=\frac{p^{\ell+1}-1}{p-1}-\ell-1$ and $i-\nu_p(i)\ge p^{\ell+1}-\ell-1$, implying the
lemma.
\end{proof}

The following lemma is easily proved by  induction on $A$.
\begin{lem}\label{lem2} If $A$ and $B$ are positive integers, then
$$\sum_{i=0}^{A-1}\tbinom{i+B-1}i=\tbinom{A+B-1}B.$$
\end{lem}

Now we can prove Theorem \ref{thm4}. We first prove it when $p=2$, and then indicate the minor changes required when $p$ is odd.
Using Theorem \ref{CM2} at the first step and Theorem \ref{CMthm} at the second, we have
\begin{eqnarray*}&&S(2^ea+d-1,2^eb+d)\\
&\equiv&\sum_{i=2^eb}^{2^ea-1}S(i,2^eb)S(2^ea+d-1-i,d)\mod 2^e\\
&\equiv&\sum_{j=2^{e-1}b}^{2^{e-1}a-1}\binom{j-2^{e-2}b-1}{j-2^{e-1}b}S(2^ea+d-1-2j,d)\mod 2^{e-1}\\
&=&\sum_{k=0}^{2^{e-1}(a-b)-1}\binom{k+2^{e-2}b-1}kS(2^e(a-b)+d-1-2k,d)\\
&=&\sum_{\ell=1}^{2^{e-1}(a-b)}\binom{2^{e-2}(2a-b)-1-\ell}{2^{e-2}b-1}S(2\ell+d-1,d)\\
&=&\sum_{\ell=1}^{2^{e-1}(a-b)}\binom{2^{e-2}(2a-b)-1-\ell}{2^{e-2}b-1}\bigl(T_2(2\ell+d-1,d)\pm\frac1{d!}\sum_j\binom d{2j}(2j)^{2\ell+d-1}\bigr).
\end{eqnarray*}

We have $\nu_2\binom{2^{e-2}(2a-b)-1-\ell}{2^{e-2}b-1}=f(a,b)+e-\nu_2(\ell)$, where $f(a,b)=\nu_2\binom{2a-b-1}{2a-2b}+\nu_2(a-b)-1$.
By \cite[Thm 1.5]{DS},
\begin{equation}\label{DSeq}\nu_2\bigl(\tfrac1{d!}\sum\tbinom d{2j}(2j)^{2\ell+d-1}\bigr)\ge2\ell+\tfrac d2-1.\end{equation}
Thus, using  Lemma \ref{Tdiff} at the first step and Lemma \ref{lem2} at the second,  we obtain
\begin{eqnarray*}&&S(2^ea+d-1,2^eb+d)\\
&\equiv& T_2(d-1,d)\sum_{k=0}^{2^{e-1}(a-b)-1}\binom{k+2^{e-2}b-1}k\mod 2^{\min(e-1,e+f(a,b)-\lg(d))}\\
&=&T_2(d-1,d)\binom{2^{e-1}(a-b)+2^{e-2}b-1}{2^{e-2}b}.
\end{eqnarray*}
Letting $e\to\infty$ yields the claim of Theorem \ref{thm4}. In the congruence, we have also used that $\nu_2(T_2(d-1,d))\ge0$. In fact, by (\ref{DSeq}) and
$S(d-1,d)=0$, we have $\nu_2(T_2(d-1,d))\ge\frac d2-1$. See Table \ref{t2} for some explicit values of $T_2(d-1,d)$.

We now present the minor modifications required when $p$ is odd and $a\equiv b$ mod $(p-1)$. Let $a=b+(p-1)t$. Then
\begin{eqnarray*} &&S(p^ea+d-1,p^eb+d)\\
&\equiv&\sum_{j=0}^{p^et-1}S(p^eb+(p-1)j,p^eb)S(p^e(a-b)-(p-1)j+d-1,d)\\
&\equiv&\sum_{j=0}^{p^et-1}\binom{p^{e-1}b+j-1}jS(p^e(p-1)t-(p-1)j+d-1,d)\\
&=&\sum_{\ell=1}^{p^et}\binom{p^et+p^{e-1}b-\ell-1}{p^{e-1}b-1}S((p-1)\ell+d-1,d)\\
&\equiv&T_p(d-1,d)\sum_{j=0}^{p^et-1}\binom{p^{e-1}b+j-1}j\\
&=&T_p(d-1,d)\binom{p^et+p^{e-1}b-1}{p^{e-1}b}. \end{eqnarray*}

\section{More formulas and numerical values}\label{messysec}
In Theorem \ref{thm3}, we gave a simple formula for $S(p^\infty a+c,p^\infty b+d)$ when $d\le0$. For $d>0$, all values can be written
explicitly using (\ref{stirfor}) and the initial values given in Theorem \ref{thm4}, provided $a\equiv b$ mod $(p-1)$.

First assume $c\ge d-1$. For $i\ge1$, define Stirling-like numbers $S_i(c,d)$ satisfying that for
$d<i$ or $c\le d-1$ the only nonzero value is $S_i(i-1,i)=1$ and satisfying the analogue of (\ref{stirfor}) when $c\ge d$. Note that $S_1(c,d)=S(c,d)$ if $(c,d)\not\in\{(0,0),(0,1)\}$.
The following result is easily obtained. Here we use that the binomial coefficient in Theorem \ref{thm4} equals $\frac p{p-1}\frac{a-b}bS(p^\infty a,p^\infty b)$.
\begin{prop} Assume $a\equiv b\mod (p-1)$. For $d\ge1$, $c\ge d-1$, we have
$$S(p^\infty a+c,p^\infty b+d)=S(p^\infty a,p^\infty b)\bigl(S(c,d)+\sum_{i=1}^d S_i(c,d)T_p(i-1,i)\tfrac p{p-1}\tfrac{a-b}b\bigr).$$   \label{p1}
\end{prop}

The reader may obtain a better feeling for these numbers from the  table of values of
$S(p^\infty a+c,p^\infty b+d)/S(p^\infty a,p^\infty b)$
in Table \ref{t1}, in which
 $T_i$ denotes $T_p(i-1,i)\frac p{p-1}\frac{a-b}b$.

\begin{table}[H]
\caption{$S(p^\infty a+c,p^\infty b+d)/S(p^\infty a,p^\infty b)$ when $a\equiv b\mod(p-1)$}
\label{t1}
\begin{tabular}{cr|ccccc}
&&&&$d$&&\\
&&$1$&$2$&$3$&$4$&$5$\\
\hline
&$0$&$T_1$&&&&\\
\hline
&$1$&$1+T_1$&$T_2$&&&\\
\hline
$c$&$2$&$1+T_1$&$1+T_1+2T_2$&$T_3$&&\\
\hline
&$3$&$1+T_1$&$3+3T_1+4T_2$&$1+T_1+2T_2$&$T_4$&\\
&&&&$+3T_3$&&\\
\hline
&$4$&$1+T_1$&$7+7T_1+8T_2$&$6+6T_1$&$1+T_1+2T_2$&$T_5$\\
&&&&$+10T_2+9T_3$&$+3T_3+4T_4$&\\[0.5ex]
\hline
&$5$&$1+T_1$&$15+15T_1$&$25+25T_1$&$10+10T_1+18T_2$&$1+T_1+2T_2$\\
&&&$+16T_2$&$+38T_2+27T_3$&$+21T_3+16T_4$&$+3T_3+4T_4$\\
&&&&&&$+5T_5$
\end{tabular}
\end{table}

The first few values of $T_2(d-1,d)$ and $T_3(d-1,d)$ are given in Table \ref{t2}.
\begin{table}[H]
\caption{Some values of $T_2(d-1,d)$ and $T_3(d-1,d)$}
\label{t2}
\begin{tabular}{c|cccccccc}
\hline
$d$&$1$&$2$&$3$&$4$&$5$&$6$&$7$&$8$\\
\hline
$T_2(d-1,d)$&$1$&$-1$&$2$&$-\frac{14}3$&$12$&$-\frac{164}5$&$\frac{4208}5$&$-\frac{86608}{315}$\\[0.5ex]
$T_3(d-1,d)$&$1$&$0$&$-\frac32$&$\frac92$&$-\frac{27}4$&$-\frac{81}{20}$&$\frac{4779}{80}$&$-\frac{15309}{80}$\\[0.5ex]
\hline
\end{tabular}
\end{table}

For $c<d-1$, we use (\ref{stirfor}) to work backwards from $S(p^\infty a+d-1,p^\infty b+d)$, obtaining
\begin{prop} Suppose $a\equiv b\mod(p-1)$. For $k\ge 1$, $d\ge 0$, let $Y(k,d)=S(p^\infty a+d-k,p^\infty b+d)$. Then $Y(1,d)$ is as in Theorem
\ref{thm4}
for $d\ge1$, $Y(k,0)=0$ for $k\ge1$, and, for $k\ge2$, $d\ge1$,
$$Y(k,d)=\bigl(Y(k-1,d)-Y(k-1,d-1)\bigr)/d.$$      \label{Yprop}
\end{prop}

We illustrate these values in Table \ref{t3}, where again $T_i$ denotes $T_p(i-1,i)\frac p{p-1}\frac{a-b}b$.
\begin{table}[H]
\caption{$S(p^\infty a+c,p^\infty b+d)/S(p^\infty a,p^\infty b)$ when $a\equiv b\mod(p-1)$}
\label{t3}
\begin{tabular}{cr|cccc}
&&&$d$&&\\
&&$1$&$2$&$3$&$4$\\
\hline
&$-2$&$T_1$&$\tfrac{1}8T_2-\tfrac{7}8T_1$&$\tfrac{1}{81}T_3-\tfrac{65}{648}T_2+\tfrac{85}{216}T_1$&$\tfrac{1}{1024}T_4-\tfrac{781}{82944}T_3+\tfrac{865}{20736}T_2-\tfrac{415}{3456}T_1$\\
%&&&&$+\tfrac{85}{216}T_1$&$+\tfrac{865}{20736}T_2-\tfrac{415}{3456}T_1$\\
\hline
&$-1$&$T_1$&$\tfrac{1}4T_2-\tfrac{3}4T_1$&$\tfrac{1}{27}T_3-\tfrac{19}{108}T_2+\tfrac{11}{36}T_1$&$\tfrac{1}{256}T_4-\tfrac{175}{6912}T_3+\tfrac{115}{1728}T_2-\tfrac{25}{288}T_1$\\
%&&&&$+\tfrac{11}{36}T_1$&$+\tfrac{115}{1728}T_2-\tfrac{25}{288}T_1$\\
\hline
$c$&$0$&$T_1$&$\tfrac{1}2T_2-\tfrac{1}2T_1$&$\tfrac{1}9T_3-\tfrac{5}{18}T_2+\tfrac{1}6T_1$&$\tfrac{1}{64}T_4-\tfrac{37}{576}T_3+\tfrac{13}{144}T_2-\tfrac{1}{24}T_1$\\
%&&&&$+\tfrac{1}6T_1$&$+\tfrac{13}{144}T_2-\tfrac{1}{24}T_1$\\
\hline
&$1$&&$T_2$&$\tfrac{1}3T_3-\tfrac{1}3T_2$&$\tfrac{1}{16}T_4-\tfrac{7}{48}T_3+\tfrac{1}{12}T_2$\\
\hline
&$2$&&&$T_3$&$\tfrac{1}4T_4-\tfrac{1}4T_3$
\end{tabular}
\end{table}

Note that since $S(d-1,d)=0$ and $T_p(n,k)-S(n,k)$ is a sum like that in (\ref{Tdef}) taken over $i\equiv0$ mod $p$, we deduce
that $T_p(d-1,d)=0$ if $1<d<p$, which simplifies these results slightly.

 \section{The case $a\not\equiv b$ mod $(p-1)$}       \label{sec4}
 In this section, we complete the proof of Theorem \ref{thm1}  when $a\not\equiv b$ mod $(p-1)$ by proving the following case.
 \begin{thm}\label{aneb} Suppose $0\le b\le a$ and $d\ge1$. Then the $p$-adic limit of $S(p^{e+1}a-(a-b),p^{e+1}b+d)$ exists as $e\to\infty$.
 \end{thm}
 \noindent Then $\ds\lim_e S(p^{e+1}a+c,p^{e+1}b+d)$ exists for all integers $c$ by induction using (\ref{stirfor}).

 Let $R_p(e)=(p^{e+1}-1)/(p-1)$. The proof of Theorem \ref{aneb} begins with, mod $p^e$,
 \begin{eqnarray*}&&S(p^{e+1} a-(a-b),p^{e+1}b+d)\\
&\equiv&\sum_{j=0}^{R_p(e)(a-b)}S(p^{e+1}b+(p-1)j,p^{e+1}b)S((p^{e+1}-1)(a-b)-(p-1)j,d)\\
&\equiv&\sum_{j=0}^{R_p(e)(a-b)}\binom{p^eb+j-1}j\frac{(-1)^d}{d!}\sum_{i=0}^d(-1)^i\binom dii^{(p^{e+1}-1)(a-b)-(p-1)j}\\
&=&\sum_{i=0}^d(-1)^{i+d}\frac1{d!}\binom di\sum_{j=0}^{R_p(e)(a-b)}\binom{p^eb+j-1}ji^{(p^{e+1}-1)(a-b)-(p-1)j}.\end{eqnarray*}
We show that for each $i$, the limit as $e\to\infty$ of
\begin{equation}\label{trm}\sum_{j=0}^{R_p(e)(a-b)}\binom{p^eb+j-1}ji^{(p^{e+1}-1)(a-b)-(p-1)j}\end{equation}
exists in $\Z_p$. This will complete the proof of the theorem.

If $i\not\equiv0$ mod $p$, write $i^{p-1}=Ap+1$, using Fermat's Little Theorem. Then (\ref{trm}) becomes
\begin{eqnarray*}&&\sum_{\ell=0}^{R_p(e)(a-b)}(Ap)^\ell\sum_{j=0}^{R_p(e)(a-b)}\binom{p^eb+j-1}j\binom{R_p(e)(a-b)-j}\ell\\
&=&\sum_{\ell=0}^{R_p(e)(a-b)}(Ap)^\ell\binom{p^eb+R_p(e)(a-b)}{p^eb+\ell}\end{eqnarray*}
by \cite[p.9(3c)]{Rior}.  Lemma \ref{limlem} says that for each $\ell$, there exists a $p$-adic integer
$$z_\ell:=\lim_{e\to\infty}\binom{p^eb+R_p(e)(a-b)}{p^eb+\ell}.$$
Then ${\ds\sum_{\ell=0}^\infty(Ap)^{\ell}z_\ell}$ is a $p$-adic integer, which is the limit of (\ref{trm}) as $e\to\infty$.

If $i=0$, since $0^0=1$ in (\ref{trm}) and the equations preceding it, (\ref{trm}) becomes
$$\binom{p^eb+R_p(e)(a-b)-1}{p^eb-1}=\frac{p^eb}{p^eb+R_p(e)(a-b)}\binom{p^eb+R_p(e)(a-b)}{p^eb}.$$
Since by the proof of Lemma \ref{limlem} $\nu_p\binom{p^eb+R_p(e)(a-b)}{p^eb}$ is eventually constant, $\binom{p^eb+R_p(e)(a-b)-1}{p^eb-1}\to0$ in $\Z_p$, due to the $p^eb$ factor.

We complete the proof of Theorem \ref{aneb} in the following lemma, which shows that the $p$-adic limit of (\ref{trm}) is 0 when $i\equiv0$ mod $p$ and $i>0$.
\begin{lem}\label{plem} If $0\le j\le R_p(e)(a-b)$, then
$$\nu_p\binom{p^eb+j-1}j+(p^{e+1}-1)(a-b)-(p-1)j\ge e-\log_p(a-b+p)$$
for $e$ sufficiently large.\end{lem}
\begin{proof} Let $\ell=R_p(e)(a-b)-j$ and $a-b=(p-1)t+\Delta$, $1\le\Delta\le p-1$.
The $p$-exponent of the binomial coefficient becomes
\begin{equation}\label{dpeq}d_p(b-1)+e+d_p((p^{e+1}-1)t+R_p(e)\Delta-\ell)-d_p((p^{e+1}-1)t+R_p(e)\Delta+p^eb-\ell-1).\end{equation}
Choose $s$ minimal so that $\frac{\Delta}{p-1}(p^s-1)-\ell-1-t\ge0$. Then, if $e>s$, the $p$-ary expansion of $(p^{e+1}-1)t+R_p(e)\Delta-\ell$ splits as
$$p^e(pt+\Delta)\quad + \quad p^s\frac{p^{e-s}-1}{p-1}\Delta\quad + \quad \frac{p^s-1}{p-1}\Delta-\ell-t,$$
and there is a similar splitting for the expression at the end of (\ref{dpeq}). We obtain that (\ref{dpeq}) equals
$$e+\nu(b)+\nu\tbinom{pt+b+\Delta}b-\nu_p\bigl(\tfrac{\Delta}{p-1}(p^s-1)-\ell-t\bigr).$$
The expression in the lemma equals this plus $(p-1)\ell$. Since $s$ was minimal, we have
$\frac{\Delta}{p-1}(p^s-1)-\ell-t\le (p-1)(\ell+t)+p+\Delta$, and hence $\nu_p(\frac{\Delta}{p-1}(p^s-1)-\ell-t)\le\log_p((p-1)(\ell+t)+p+\Delta)$.
The smallest value of  $(p-1)\ell-\log_p((p-1)(\ell+t)+p+\Delta)$ occurs when $\ell=0$.  We obtain that the expression in the lemma is $\ge e-\log_p(a-b+p)$.
\end{proof}

The following lemma was referred to above.
\begin{lem}\label{limlem} If $\a$ and $b$ are positive integers and $\ell\ge 0$, then  $$\lim_{e\to\infty}\binom{p^eb+R_p(e)\a}{p^eb+\ell}$$
exists in $\Z_p$.\end{lem}
The proof of the lemma breaks into two parts: showing that the $p$-exponents are eventually constant, and showing that the unit parts
approach a limit.

The proof that the $p$-exponent is eventually constant is very similar to the proof of Lemma \ref{plem}. Let $\a=(p-1)t+\Delta$ with $1\le\Delta\le p-1$, and choose
$s$ minimal such that $\frac{\Delta}{p-1}(p^s-1)-t-\ell\ge0$. Then the $p$-ary expansions split again into three parts and we obtain that
for $e>s$, the desired $p$-exponent equals $\nu_p\binom{pt+b+\Delta}b+\nu_p\binom{\Delta(p^s-1)/(p-1)-t}\ell$, independent of $e$.

We complete the proof of Lemma \ref{limlem} by showing that, if $\ell<\min(R_p(e-1)\a,p^eb)$, then
\begin{equation}\label{ustmt}\u_p\binom{p^{e-1}b+R_p(e-1)\a}{p^{e-1}b+\ell}\equiv\u_p\binom{p^eb+R_p(e)\a}{p^eb+\ell}\mod p^{e+f(\a,b,\ell)-1},\end{equation}
where $f(\a,b,\ell)=\min(\nu_p(b)-\lg_p(\a),\nu_p(\a)-\lg_p(\ell),\nu_p(b)-\lg_p(\ell),1)$. We write the second binomial coefficient in (\ref{ustmt}) as
\begin{equation}\label{4facs}(-1)^{eb}\frac{(p^eb+R_p(e)\a)!}{(R_p(e)\a)!}\cdot\frac{(R_p(e)\a)!}{(R_p(e)\a-\ell)!}\cdot\frac{(p^eb)!}{(p^eb+\ell)!}\cdot\frac{(-1)^{eb}}{(p^eb)!}.\end{equation}
We show that these four factors are congruent to their $(e-1)$-analogue mod $p^{e+\nu_p(b)-\lg(\a)-1}$, $p^{e+\nu_p(\a)-\lg_p(\ell)-1}$, $p^{e+\nu_p(b)-\lg_p(\ell)-1}$, and
$p^e$, respectively, which will imply the result. For the fourth factor, this was shown in \cite{Dinf}. For the second and third, the claim is clear, since each of the $\ell$ unit factors being multiplied will be congruent to their $(e-1)$-analogue modulo the specified amount.

For the first, we will prove
\begin{equation}\label{Ueq}\u_p\biggl(\frac{(R_p(e)\a+1)\cdots(R_p(e)\a+p^eb)}{(R_p(e-1)\a+1)\cdots(R_p(e-1)\a+p^{e-1}b)}\biggr)\equiv(-1)^b\mod p^{e+\nu_p(b)-\lg_p(\a)-1}.\end{equation}
Since $\u_p(j)=\u_p(pj)$, we may cancel most multiples of $p$ in the numerator with factors in the denominator. Using that $p\cdot R_p(e-1)=R_p(e)-1$, we obtain that the LHS of (\ref{Ueq}) equals $P\u_p(A)/\u_p(B)$, where $P$ is the product of the units in the numerator, $A$ is the product of all $j\equiv0$ mod $p$ which satisfy
$$(R_p(e)-1)\a+p^eb< j\le R_p(e)\a+p^eb,$$
and $B$ is the product of all integers $k$ such that
\begin{equation}\label{range}R_p(e-1)\a+1\le k\le R_p(e-1)\a+\bigl[\tfrac{\a}p\bigr].\end{equation}

Since the mod $p^e$ values of the $p$-adic units in any interval of $p^e$ consecutive integers are just a permutation of the set of positive $p$-adic units less than $p^e$, and by \cite[Lemma 1]{Gr} the product of these is $-1$ mod $p^e$, we obtain $P\equiv(-1)^b\mod p^e$. Thus (\ref{Ueq}) reduces to showing $\u_p(A)/\u_p(B)\equiv 1\mod p^{e+\nu_p(b)-\lg_p(\a)-1}$.

We have
$$\frac{\u_p(A)}{\u_p(B)}=\prod\frac{\u_p(k+p^{e-1}b)}{\u_p(k)},$$
taken over all $k$ satisfying (\ref{range}). We show that if $k$ satisfies (\ref{range}), then
\begin{equation}\label{nuk}\nu_p(k)\le\lg_p(\a).\end{equation}
Then $\u_p(k)\equiv\u_p(k+p^{e-1}b)\mod p^{e+\nu_p(b)-\lg_p(\a)-1}$, establishing the result.

We prove (\ref{nuk}) by showing that it is impossible to have $1\le\a<p^t$, $1\le i\le[\frac{\a}p]$, and
\begin{equation}\label{R}R_p(e-1)\a+i\equiv0\mod p^t.\end{equation}
From (\ref{R}) we deduce $\a\equiv i(p-1)\mod p^t$. But $i(p-1)<\a$, so the only way to satisfy (\ref{R})
would be with $\a=p^t$ and $i=0$, but $\a<p^t$.

\section{Another kind of $p$-adic Stirling number}\label{sec5}
It is well-known (see, e.g., \cite{Kw}) that, if $p$ is any prime and $y\equiv0$ mod $p-1$, then
$$\nu_p(S(x+y,k)-S(x,k))\ge\nu_p(y)+2-\lceil\log_p(k)\rceil,$$
provided that $x$ and $x+y$ are greater than $k$.
This implies that for $0\le i\le p-2$, there is a continuous function  $f_{i,k}:\Z_p\to\Z_p$ such that
$f_{i,k}(m)=S(i+m(p-1),k)$ for all  integers $m$ such that $i+m(p-1)\ge k$. That is, it defines $S(x,k)$ for any $p$-adic integer $x$.
See \cite[p.73]{Cl} for a related discussion.
In \cite{Cl}, the idea of finding $p$-adic integers $z$ which are zeros of these functions (i.e., $f_{i,k}(z)=0$)
is introduced, and its study is continued in \cite{D}.

This is a quite different notion of $p$-adic Stirling number than the one introduced in our Section \ref{intro}.

\def\line{\rule{.6in}{.6pt}}

\end{document}